\documentclass[11 pt]{amsart}

\usepackage[T1]{fontenc}
\usepackage{graphicx}
\usepackage{enumerate}
\usepackage{amsmath}
\usepackage{amsthm}
\usepackage{amssymb}

\usepackage{todonotes}

\newtheorem{theorem}{Theorem}[section]
\newtheorem{lemma}[theorem]{Lemma}

\theoremstyle{definition}
\newtheorem{definition}[theorem]{Definition}

\newtheorem{corollary}[theorem]{Corollary}

\newtheorem{proposition}[theorem]{Proposition}

\theoremstyle{definition}

\numberwithin{equation}{section}

\newcommand{\ml}{Martin-L\"of }

\newcommand{\piint}{[-\pi,\pi]}

\newcommand{\norm}[1]{\left|\left|#1\right|\right|}

\newcommand{\Poisson}[1]{P[#1](x,y)}

\newcommand{\Q}{\mathbb{Q}}
\newcommand{\C}{\mathbb{C}}
\newcommand{\R}{\mathbb{R}}
\newcommand{\N}{\mathbb{N}}
\newcommand{\Z}{\mathbb{Z}}
\newcommand{\I}{\mathcal{I}}

\newcommand{\UHP}{\mathrm{UHP}}

\newcommand{\st}{\;|\;}

\title{Algorithmic randomness in harmonic analysis}

\author[Franklin]{Johanna N.Y.\ Franklin}
\address[Franklin]{Department of Mathematics \\ Room 306, Roosevelt Hall \\ Hofstra University \\ Hempstead, NY 11549-0114 \\ USA}
\email{johanna.n.franklin@hofstra.edu}
\urladdr{http://people.hofstra.edu/Johanna\_N\_Franklin/}

\author[Rodriguez]{Lucas E.\ Rodriguez}
\address[Rodriguez]{Department of Mathematics \\ University of Dallas \\ Irving, TX 775062-2206 \\ USA}
\email{lerodriguez@udallas.edu}

\author[Rojas]{Diego A.\ Rojas}
\address[Rojas]{Department of Mathematics and Statistics \\ Sam Houston State University \\ Huntsville, TX 77341-2206 \\ USA}
\email{dar117@shsu.edu}
\begin{document}

\begin{abstract}
    Within the last fifteen years, a program of establishing relationships between algorithmic randomness and almost-everywhere theorems in analysis and ergodic theory has developed. In harmonic analysis, Franklin, McNicholl, and Rute characterized Schnorr randomness using an effective version of Carleson’s Theorem. We show here that, for computable $1<p<\infty$, the reals at which the Fourier series of a weakly computable vector in $L^p\piint$ converges are precisely the \ml random reals. Furthermore, we show that radial limits of the Poisson integral of an $L^1(\R)$-computable function coincide with the values of the function at exactly the Schnorr random reals and that radial limits of the Poisson integral of a weakly $L^1(\R)$-computable function coincide with the values of the function at exactly the \ml random reals.
\end{abstract}

\maketitle

\section{Introduction}

The relationship between algorithmic randomness notions and classical results in analysis has been extensively studied. Theorems in analysis often hold on a conull set, and any plausible randomness notion holds of a conull set as well. Therefore, one can often obtain theorems of the following form:

\begin{quote}
    An element of $[0,1]$ if $\mathcal{R}$-random if and only if it satisfies Theorem $T$.
\end{quote}

Theorems of this sort have been obtained for Birkhoff's Ergodic Theorem and the Poincar\'{e} Recurrence Theorem \cite{bdhms12,fgmn12,ft-mp,hoyrup13,rutethesis,v97}, differentiability \cite{rutethesis,fkns14,mnz16}, Brownian motion \cite{fm23,f15,fmd14}, and Fourier series \cite{fmr19}. Here, we turn our attention to two results in harmonic analysis: Carleson's Theorem in Fourier analysis \cite{carleson66,hunt68} and solutions to the Dirichlet problem for the upper half-plane \cite{NRS82}.

Carleson's Theorem is stated as follows:
\begin{quote}
    Suppose $1<p<\infty$. If $f$ is a function in $L^p\piint$, then the Fourier series of $f$ converges to $f$ almost everywhere.
\end{quote}
Franklin, McNicholl, and Rute demonstrated a relationship between the convergence of the Fourier series of a function at a point and whether the point is Schnorr random \cite{fmr19}.\footnote{The result is not exactly an equivalence due to slightly stronger conditions on the function in one of the directions.} In this paper, we demonstrate an equivalence between the convergence of the Fourier series of a function at a point and whether it is \ml random.

The Dirichlet problem for the upper half-plane can be stated as follows:
\begin{quote}
    Let $\UHP\subseteq\R^2$ denote the upper half-plane. Given a function $f$ (known as \emph{boundary data}) that has values everywhere on $\partial\UHP$, is there a unique continuous function $u$ twice continuously differentiable on $\UHP$ and continuous on $\partial\UHP$ such that $u$ is harmonic on $\UHP$ and $u=f$ on $\partial\UHP$?
\end{quote}
Given boundary data $f\in L^1(\R)$, the \emph{Poisson integral $P[f]$} of $f$, given by
\[
P[f](x,y)=\int_{\R}\frac{1}{\pi}\frac{y}{(x-t)^2+y^2}f(t) \ dt
\]
for all $(x,y)\in\UHP$, solves the Dirichlet problem for the upper half-plane. If $f$ is continuous and compactly supported, then $P[f]$ is a \emph{classical solution} since $\lim_{y\rightarrow0^+}P[f](x,y)=f(x)$ for every $x\in\R$. Otherwise, $P[f]$ is a \emph{weak solution} in the sense that $\lim_{y\rightarrow0^+}P[f](x,y)=f(x)$ for \emph{almost every} $x\in\R$. In this paper, we characterize Schnorr randomness via radial limits of the Poisson integral of integrable boundary data.

\section{Background}

Our notation is standard and follows \cite{soare}; note that $\lambda$ is Lebesgue measure, $\chi_A$ is the characteristic function of the set $A$, and $C_c(\R)$ is the space of compactly supported continuous functions. For a general introduction to computability theory, we refer the reader to \cite{o1,o2}, and for general references on randomness, we refer the reader to \cite{dhbook} and \cite{fpbook}.

\subsection{Computable analysis and randomness}

We will discuss \ml randomness and Schnorr randomness in this paper.

\begin{definition}\normalfont \cite{ml66,schnorr}
    A \emph{\ml test} is a sequence $\{V_k\}_k$ of uniformly $\Sigma^0_1$ classes in Cantor space such that for each $k$, we have $\lambda(V_k)\leq 2^{-k}$. A real $x$ is \emph{\ml random} if there is no \ml test $\{V_k\}_k$ such that $x\in \cap V_k$.

    A \emph{Schnorr test} is a \ml test such that $\lambda(V_k)$ is computable uniformly in $k$, and a real $x$ is \emph{Schnorr random} if there is no Schnorr test $\{V_k\}_k$ such that $x\in \cap V_k$. 
\end{definition}

We observe that, while these randomness notions were originally defined in the context of Cantor space, the definitions are easily transferable to any computable probability space.

Every \ml random real is clearly Schnorr random, but the converse does not hold. Without loss of generality, we will consider all \ml tests to be nested; that is, for every $i$, we will have $V_i \supseteq V_{i+1}$.

We will also make use of the integral test characterization of \ml randomness due to Miyabe \cite{miyabe13}:
\begin{definition}\normalfont
    A \emph{\ml integral test} is a lower semicomputable function $T:\piint\rightarrow [0,\infty]$ such that $\int^\pi_{-\pi} T\ d\lambda < \infty$. A point $x\in\piint$ is \ml random if and only if $T(x)<\infty$ for every \ml integral test $T$.
\end{definition}

Now we turn our attention to computable analysis.
A \emph{computable Polish space} is a triple $(X,d,\alpha)$ with the following properties:
\begin{enumerate}[(1)]
    \item $(X,d)$ is a complete separable metric space;
    \item $\alpha:\N\rightarrow X$ is an enumeration of a countable subset of $X$ whose range is dense in $X$;
    \item $d(\alpha_i,\alpha_j)$ is computable uniformly in $i,j$.
\end{enumerate}
We refer to $\alpha(i)$ for each $i\in\N$ as the \emph{rational} or \emph{ideal points} of $X$.

A \emph{computable Banach space} over $\C$ is a triple $(X,\norm{\cdot},e)$ with the following properties:
\begin{enumerate}[(1)]
    \item $(X,\norm{\cdot})$ is a Banach space over $\C$;
    \item $e:\N\rightarrow X$ is an enumeration of a countable subset of $X$ whose linear span is dense in $X$ and is denoted by $\{e_k\}_k$;
    \item $\norm{\sum_{i=0}^kc_ke_{n_k}}$ is computable uniformly in $\langle k, \langle n_0,\ldots,n_k\rangle\rangle$ where $c_k\in\Q[i]$ for each $k$.
\end{enumerate}
We refer to the $e_k$s as the \emph{distinguished vectors} of $X$. Furthermore, for each $k\in\N$ and each $n_0,\ldots,n_k\in\N$, $\sum_{i=0}^kc_ke_{n_k}$ is called a \emph{rational vector} in $X$ when $c_0,\ldots,c_k\in\Q$.



In the following definitions of Fourier coefficients and Fourier partial sums, we assume that $p\geq 1$ is a computable real.
For all $n\in\Z$ and $f\in L^p\piint$, we say
$$c_n(f) = \frac{1}{2\pi}\int^\pi_{-\pi} f(t)e^{int}\ dt,$$
and for all $f\in L^p\piint$ and $N\in\N$, we say
$$S_N(f)=\sum^N_{n=-N} c_n(f)e^{int}.$$

\begin{definition}\normalfont
    A \emph{trigonometric polynomial} is a function in the linear span of $\{e^{int} \st n\in \Z\}$, and the degree of such a polynomial $p$ is the smallest $d\in\N$ such that $S_d(p)=p$. We call a trigonometric polynomial \emph{rational} if the coefficients come from $\Q[i]$.
\end{definition}

Note that $(L^p\piint, \norm{\cdot}_1, e)$ where $e(n)(t)=e^{int}$ for all $t\in[-\pi,\pi]$ is a computable Banach space over $\C$. We call $f\in L^p\piint$ a \emph{computable vector in $L^p\piint$} if there is a computable sequence of rational trigonometric polynomials $\{\tau_n\}_n$ that converges to $f$ with a computable modulus of convergence in $L^p$.

\begin{definition}\normalfont
    Suppose $\{f_n\}_n$ is a sequence of functions on $\piint$. A function $\eta:\N\times\N\rightarrow \N$ is a \emph{modulus of almost everywhere convergence} for $\{f_n\}_n$ if for all $k,m$ we have
    \[
    \lambda(\{t\in\piint \st \exists M,N \geq \eta(k,m) |f_N(t)-f_M(t)|\geq 2^{-k} \}) < 2^{-m}.
    \]
\end{definition}

Now we can define the type of function we will consider in Section \ref{sec:FourierSeries}.

\begin{definition}\normalfont
    A function $f\in L^p\piint$ is a \emph{weakly computable vector in $L^p$} if there is a computable sequence of rational trigonometric polynomials $\{\tau_n\}_n$ such that $f=\lim_n\tau_n$ and $\sum_n\|\tau_n-\tau_{n+1}\|<\infty$.
\end{definition}

In Section \ref{sec:DirichletProblem}, we instead imbue $L^1(\R)$ with a computable Polish space structure. Let $d_{\norm{\cdot}_1}$ denote the metric induced by the 1-norm. Then $(L^1(\R), \linebreak d_{\norm{\cdot}_1}, \alpha)$, with $\alpha$ being an enumeration of all piecewise linear functions with rational vertices from $C_c(\R)$, forms a computable Polish space (i.e., the space of compactly supported continuous real-valued functions). An \emph{$L^1(\R)$-computable} function is a computable point of $(L^1(\R), d_{\norm{\cdot}_1}, \alpha)$. Equivalently, $f$ is $L^1(\R)$-computable if there is a computable sequence $\{ p_n\}_n$ of piecewise linear functions with rational vertices from $C_c(\R)$ that converges to $f$ with a computable modulus of convergence in $L^1$.

The following lemma is an adaptation of Lemma 3.6 in \cite{Pathak.Rojas.Simpson:2014} to the space $(L^1(\R), d_{\norm{\cdot}_1}, \alpha)$. The proof follows \emph{mutatis mutandis}.

\begin{lemma}\label{lm:simple}
Let $f\in L^1(\R)$ be $L^1(\R)$-computable. Let $\{f_n\}_n$ be a sequence of piecewise linear functions with rational vertices from $C_c(\R)$ that approximates $f$. Then we can find a sequence of uniformly $\Sigma^0_1$ sets $\{V_k\}_k$ such that the following statements hold:
    \begin{enumerate}[(1)]
    \item $\lambda (V_k) \leq \frac{2+\sqrt{2}}{2^{k-1}}$;
    \item The sequence $\lambda (V_k)$ is uniformly computable;
    \item when $x \not\in V_k$ and $n \geq k$ we have
    \[ 
    |f_i(x) - f_{2n} (x)| \leq \frac{2+\sqrt{2}}{2^n}
    \]
    for all $i \geq 2n$.
    \end{enumerate}
\end{lemma}

\subsection{Classical analysis}

The following definitions and result will be useful to us in Section \ref{sec:FourierSeries}. The \emph{Dirichlet kernel} is the collection $\{D_N\}_{N\in\N}$ of trigonometric polynomials given by
\[
D_N(x)=\sum_{|m|\leq N}e^{2\pi i m x}.
\]
Note that for any $f\in L^p\piint$, $D_N*f=S_N(f)$ for each $N\in\N$. The \emph{Fej\'{e}r kernel} is the collection $\{F_N\}_{N\in\N}$ of trigonometric polynomials given by
\[
F_N(x)=\dfrac{1}{N+1}\sum_{j=0}^ND_j(x).
\]
The following is Proposition 3.1.7 from \cite{grafakos14}.
\begin{proposition}\label{prop:fejer}
For each $N\in\N$,
\[
F_N(x)=\dfrac{1}{N+1}\left(\dfrac{\sin\left(\dfrac{(N+1)x}{2}\right)}{\sin\left(\dfrac{x}{2}\right)}\right)^2
\]
for all $x\in\piint$. Furthermore, $c_n(F_N)=1-\dfrac{|n|}{N+1}$ if $|n|\leq N$ and $c_n(F_N)=0$ otherwise.
\end{proposition}

We will need the following definition and results for Section \ref{sec:DirichletProblem}. The \emph{Poisson kernel} is the real-valued function $P_y$ given by $P_y (x) = \dfrac{1}{\pi} \dfrac{y}{x^2 + y^2}$ for all $y>0$.
The most relevant properties of the Poisson kernel for our purposes are:
\begin{enumerate}
    \item $P_y (x) \geq 0$ for all $y>0$.
    \item $P_y (x) \leq \frac{1}{\pi y}$ for each $y>0$.
    \item $\int_{\R} P_y (x) \ dx = 1$ for each $y>0$.
\end{enumerate}
The \emph{Poisson integral of $f$}, denoted $P[f](x,y)$, is given by $P[f](x,y) = \int_{\R} P_y (x-t) f(t) \ dt $.

Now we discuss an application of the Hardy-Littlewood Maximal Theorem known as the \emph{Poisson Maximal Theorem}. The original result was proven in \cite{NRS82}. Here, we present the result as given in \cite{S17}. Below, the operator $P^*$ is given by
\[ 
P^*f(x) = \sup_{y>0} P[|f|](x,y)
\]
and is called the \emph{Poisson maximal operator}.
\begin{theorem}[The Poisson Maximal Theorem]\label{thm:PMT}
There exists a positive constant $C$ such that for each $f\in L^1(\R)$ and $\alpha > 0$,
\[ 
\lambda \{P^*f > \alpha \} \leq \frac{C}{\alpha} \norm{f}_1.
\]
\end{theorem}

\begin{corollary}\label{cor:aeconv}
For all $f\in L^1(\R)$,
\[ 
\lim\limits_{y\rightarrow0^+}P[f](x,y)=f(x)
\]
for almost every $x \in \R$.
\end{corollary}
Note that, if $f\in C_c(\R)$, then $\lim_{y\rightarrow0^+}P[f](x,y)=f(x)$ for \emph{all} $x\in\R$.

\section{Convergence of Fourier Series}\label{sec:FourierSeries}

In this section, we show that the \ml random points of $\piint$ are precisely the points at which the Fourier series of a weakly computable vector in $L^p$ converges.

\begin{theorem}\label{thm:fourier.1}
    Suppose $p$ is a computable real so that $p>1$. If $t\in\piint$ is Martin-L\"{o}f random and $f$ is a weakly computable vector in $L^p\piint$, then the Fourier series of $f$ converges at $t$.
\end{theorem}

\begin{proof}
Suppose $f$ is a weakly computable vector in $L^p\piint$ and $t_0$ is \ml random. 
We begin by constructing a \ml integral test $T$ such that $T(t_0)<\infty$.

Let $T(t)=\sum_N|\tau_N(t)-\tau_{N+1}(t)|$. Note that $T$ is lower semicomputable since $\tau_N$ is computable uniformly in $N$. We show that $T$ is integrable. Let $A = \{t\in[-\pi,\pi] \st \{\tau_N(t)\}_N \ \text{converges}\}$. Since $\lambda(\piint\setminus A)=0$, it suffices to show that $\int_AT \ d\lambda<\infty$. By the Monotone Convergence Theorem and H\"{o}lder's Inequality,
\begin{align*}
\int_A T \ d\lambda &= \sum_{N=0}^{\infty}\int_A|\tau_N(t)-\tau_{N+1}(t)| \ dt \\
&\leq(2\pi)^{(p-1)/p}\sum_{N=0}^{\infty}\|\tau_N-\tau_{N+1}\|_p<\infty.
\end{align*}
Therefore, $T$ is a \ml integral test.
    
Now suppose $\{S_N(f)\}_N$ diverges at $t_0$. Then there is a $k_0$ such that 
\[
    \limsup_{M,N}|S_M(f)(t_0)-S_N(f)(t_0)|\geq2^{-k_0+1}.
\]
Let $d_N=\deg(\tau_N)$. Then, for sufficiently large $N$,
\[
|S_{d_N}(f)(t_0)-\tau_N(t_0)|=|S_{d_N}(f)(t_0)-S_{d_N}(\tau_N)(t_0)|<2^{-k_0-1}.
\]
Let $M,N$ be indices satisfying 
\begin{align*}
    |S_{d_M}(f)(t_0)-\tau_M(t_0)|&<2^{-k_0-1}; \\
    |S_{d_N}(f)(t_0)-\tau_N(t_0)|&<2^{-k_0-1}; \\
    |S_{d_M}(f)(t_0)-S_{d_{N}}(f)(t_0)|&\geq2^{-k_0+1}.
\end{align*}
By the triangle inequality,
\begin{align*}
    |\tau_M(t_0)-\tau_N(t_0)|&\geq|S_{d_M}(f)(t_0)-S_{d_{N}}(f)(t_0)|\\
    &\hspace{0.75in}-|S_{d_M}(f)(t_0)-\tau_M(t_0)|-|S_{d_N}(f)(t_0)-\tau_N(t_0)|\\
    &>2^{-k_0+1}-2\cdot2^{-k_0-1}\\
    &=2^{-k_0}.
\end{align*}
Since the above occurs for infinitely many $M$ and $N$, $\{\tau_N\}_N$ also diverges at $t_0$. 
Then there is an $k_0$ such that 
\[
    \limsup_{M,N}|\tau_M(t_0)-\tau_N(t_0)|\geq 2^{-k_0}.
\]
Choose $N_0$ such that for all $N\geq N_0$, we have $|\tau_{N}(t_0)-\tau_{N+1}(t_0)|\geq 2^{-k_0}$. It follows that
\[
T(t_0)=\sum_{N=0}^{\infty}|\tau_N(t_0)-\tau_{N+1}(t_0)|\geq\sum_{N\geq N_0}2^{-k_0}=\infty,
\]
which gives a contradiction since $t_0$ is \ml random. Therefore, $\{S_N(f)\}_N$ converges at $t_0$.

\end{proof}

Now we turn to the converse.

\begin{theorem}\label{thm:fourier.2}
    Suppose $p$ is a computable real so that $p>1$. If $t\in\piint$ is not Martin-L\"{o}f random, then there is a weakly computable vector $f$ in $L^p\piint$ such that the Fourier series of $f$ diverges at $t$.
\end{theorem}

\begin{proof}
Suppose $t$ is not \ml random. Let $\{U_n\}_n$ be a nested universal \ml test with $\lambda(U_n)\leq 2^{-n}$ and $t\in\bigcap_n U_n$.
For each $n$, let $\mathcal I_n$ be a computable enumeration of all rational closed intervals contained in $U_n$, and let $\mathcal I_n[s]$ be the first $s$ such intervals. For $I=[a,b]\in\mathcal I_n$, write $c(I)=(a+b)/2$.

Let $F_N$ be the Fej\'er kernel; that is,
\[
F_N(\theta)=
\dfrac{1}{N+1}\left(\dfrac{\sin\!\big(\frac{(N+1)\theta}{2}\big)}{\sin\!\big(\frac{\theta}{2}\big)}\right)^2
\]
for each $N\in\N$. By Proposition \ref{prop:fejer}, each $F_N$ is a nonnegative rational trigonometric polynomial whose Fourier spectrum is contained in $[-N,N]$. Moreover, there is a (computable) $A_p\geq1$ such that
\[
A_p^{-1}(N+1)^{1-1/p}\ \leq\ \|F_N\|_{L^p([-\pi,\pi])}\ \leq\ A_p(N+1)^{1-1/p}
\]
for all $N\geq0$. Also, for $|\theta|\leq \pi/(N+1)$ we have
\[
F_N(\theta)\ \geq\ \frac{4}{\pi^2}(N+1).
\]

Fix $C\in\Z^+$, and define $N_n:=\lfloor (n+1)^{\,2p+2}\rfloor$. At stage $s=2n+1$, set
\[
g_n(\theta):=\frac{C}{N_n+1}\sum_{I\in\mathcal I_n[s]} F_{N_n}\big(\theta-c(I)\big)\,,
\]
and define
\[
f_{2n}:=\sum_{j<n} g_j,\qquad f_{2n+1}:=f_{2n}+g_n,\qquad f:=\lim_{s\to\infty} f_s.
\]

Since $|\mathcal I_n[s]|=2n+1$,
\[
\|g_n\|_{p}\ \leq\ \frac{C}{N_n+1}\,(2n+1)\,\|F_{N_n}\|_{p}
\ \leq\ \frac{C A_p\,(2n+1)}{(N_n+1)^{1/p}}\ \leq\ \frac{C A_p\,(2n+1)}{(n+1)^{2+2/p}}.
\]
By the ratio test, $\sum_s\|f_{s+1}-f_s\|_{p}=\sum_n\norm{g_n}_p<\infty$. Thus, $f$ is a weakly computable vector in $L^p[-\pi,\pi]$.

Since $t\in U_n$, we can find for each $n$ an interval $I_n\in\mathcal I_n$ with $t\in I_n$. Thus,
\[
|t-c(I)|\ \leq\ \tfrac12|I|\ \le\ \tfrac12\lambda(U_n)\ \leq\ 2^{-n-1}.
\]
Since $N_n\geq n^{2p+2}$, we have $2^{-n-1}\leq \pi/(N_n+1)$ for all sufficiently large $n$. It follows that
\[
F_{N_n}\big(t-c(I)\big)\ \ge\ \frac{4}{\pi^2}(N_n+1).
\]
Thus,
\[
g_n(t)\ \ge\ \frac{C}{N_n+1}\,F_{N_n}\big(t-c(I)\big)\ \ge\ \frac{4C}{\pi^2}\ :=\ \beta\ >0
\]
for all sufficiently large $n$.

Since the Fourier spectrum of $g_n$ lies in $[-N_n,N_n]$,
\[
S_{N_n}(f)(t)-S_{N_{n-1}}(f)(t)=g_n(t)\ \ge\ \beta
\]
for all sufficiently large $n$. Hence,  $\limsup_{M,N}|S_{M}(f)(t)-S_{N}(f)(t)|=\infty$. Therefore, the Fourier series of $f$ diverges at $t$.
\end{proof}

\section{The Dirichlet Problem for the Upper Half-Plane}\label{sec:DirichletProblem}

In this section, we show that the Schnorr random reals are precisely the points at which the radial limits of the Poisson integral of $L^1(\R)$-computable boundary data coincide with the boundary data values. Then, we show that the \ml random reals are precisely the points at which the radial limits of the Poisson integral of weakly $L^1(\R)$-computable boundary data coincide with the boundary data values.

\begin{theorem}\label{thm:poisson.1}
    If $x\in\R$ is Schnorr random and $f$ is an $L^1(\R)$-computable function, then $\lim_{y\rightarrow0^+}P[f](x,y)=f(x)$.
\end{theorem}

We first need the following lemma.

\begin{lemma}\label{lm:schnorr}
Let $f\in L^1(\R)$ be $L^1(\R)$-computable, and let $\{f_n\}_n$ be a sequence of piecewise linear functions with rational vertices from $C_c(\R)$ that approximates $f$. Then we can find uniformly $\Sigma^0_1$ sets $\{U_k\}_{k}$ such that the following statements hold:
\begin{enumerate}[(1)]
    \item $\lambda (U_k) \leq \frac{3(\sqrt{2} + 2)}{2^k}$.
    \item The sequence $\lambda (U_k)$ is uniformly computable.
    \item when $x \not\in U_k$ and $n \geq k$ we have
\[ 
\int_{\R} P_y(x-t)|f(t) - f_{2n}(t)| \,dt \leq \frac{2+\sqrt{2}}{2^n}.
\]
\end{enumerate}
\end{lemma}

We can observe at this point that $\{U_k\}_{k}$ will be a Schnorr test.

\begin{proof}
(1) Suppose $f\in L^1(\R)$ is $L^1(\R)$-computable. Let $\{f_n\}_n$ be as above, and let $\epsilon > 0$ be given. Define a class of sets
\begin{align*}
S(f,\epsilon) = \left\{x \mid P^* f(x) > \epsilon  \right\} = \bigcup_ {y \in \Q^{> 0}} \left\{x \left | \int_{\R} P_y(x-t)|f(t)| \,dt > \epsilon \right. \right\}.
\end{align*}
By the Poisson Maximal Theorem (Theorem \ref{thm:PMT}), $\lambda \{S(f,\epsilon)\} \leq \frac{3}{\epsilon} \norm{f}_1$. We set $S_i = S(f_i - f_{i+1},2^{- \frac{i}{2}})$ and $U_k = \bigcup_{i=2k}^{\infty} S_i$. Clearly, $\{U_k\}_{k}$ is uniformly $\Sigma^0_1$. Now we observe that
\begin{align*}
\lambda (U_k) \leq \sum_{i=2k}^{\infty} \lambda (S_i)
\leq \sum_{i=2k}^{\infty} \frac{3}{2^{- \frac{i}{2}}} \norm{f_i - f_{i+1}}_1
\leq \sum_{i=2k}^{\infty} 3 \cdot 2^{- \frac{i}{2}}
= \frac{3(\sqrt{2} + 2)}{2^k}.
\end{align*}

(2) It is clear that $x \in S_i$ if and only if there exists a rational $y > 0$ such that $\int_{\R} P_y(x-t)|f(t)| \,dt > \epsilon$. It follows that $\lambda (S_i)$ is computable uniformly in $i$. Moreover, $\lambda (S_i) \leq 3 \cdot 2^{- \frac{i}{2}}$ implies that $\lambda (U_k)$ is computable uniformly in $k$. 

(3) Suppose $x \not\in U_k$. For each $y > 0$, for all  $n \geq k$, and for all $i \geq 2n$,
\[
\int_{\R} P_y(x-t)|f_i - f_{i+1}| \ dt \leq 2^{- \frac{i}{2}}
\]
implies that $\Poisson{|f_i - f_{i+1}|} \leq 2^{- \frac{i}{2}}$. Since $P[|f-f_{n}|]( \cdot ,y) \to 0$ pointwise, we have
\begin{align*}
\Poisson{|f-f_{2n}|} \leq \sum_{i=2n}^{\infty} \Poisson{|f_i-f_{2n}|} \leq \frac{2 + \sqrt{2}}{2^n}.
\end{align*} 
\end{proof}

As indicated above, we can see that $\{U_k\}_{k}$ is a Schnorr test. This specific Schnorr test, together with the Schnorr test from the similar Lemma \ref{lm:simple}, will be used in the proof of Theorem \ref{thm:poisson.1}.

\begin{proof}[Proof of Theorem \ref{thm:poisson.1}]
Let $W_k = U_k \cup V_k$ where $\{U_k\}_k$ is the Schnorr test from Lemma \ref{lm:schnorr} and $\{V_k\}_k$ is the Schnorr test from Lemma \ref{lm:simple}. It suffices to show that if $x \not\in \bigcap_{i=0}^{\infty} W_k$, then
\begin{align*}
\lim_{y \to 0^+} \Poisson{f} = f(x).
\end{align*}

Fix $x \not\in \bigcap_{i=0}^{\infty} W_k$. 
Then $x\not\in W_k$ for some $k$. 
Recall that $\lim_{y\rightarrow0^+}\Poisson{g}=g(x)$ for all $x\in\R$ whenever $g \in C_c(\R)$. 
Thus, choose $\delta_n$ so that $0<y<\delta_n$ implies $|\Poisson{f_{2n}} - f_{2n}(x)| < 2^{-n}$. 
Since $x \not\in V_k$, $|f_i(x) - f_{2n}(x)| \leq \frac{2+\sqrt{2}}{2^n}$ for all $i \geq 2n$. 
By definition of $f$, we can choose $i_0\geq 2n$ so large that $|f_{i_0}(x)-f(x)|<2^{-n}$. 
Moreover, since $x \not\in U_k$, $\Poisson{|f-f_{2n}|} \leq \frac{2+\sqrt{2}}{2^n}$. 
Hence, $0<y<\delta_n$ implies the following.
\begin{align*}
|\Poisson{f} - f(x)| & \leq \Poisson{|f-f_{2n}|} \\
&\hspace{0.5in}+ |\Poisson{f_{2n}} - f_{2n}(x)| \\
&\hspace{1 in} + |f_{2n}(x) - f_{i_0}(x)| \\
&\hspace{1.5 in}+|f_{i_0}(x)-f(x)|\\
&\leq \frac{2+\sqrt{2}}{2^n} + \frac{1}{2^n} + \frac{2+\sqrt{2}}{2^n} +\dfrac{1}{2^n}\\
&=\frac{6+2 \sqrt{2}}{2^n}.
\end{align*}

Therefore,
\[
\lim_{y\rightarrow0^+}P[f](x,y)=f(x).
\]
\end{proof}

Now we turn to the converse.

\begin{theorem}\label{thm:poisson.2}
    If $x\in\R$ is not Schnorr random, then there is an $L^1(\R)$-comput-able function $f$ such that $\lim_{y\rightarrow0^+}P[f](x,y)\neq f(x)$.
\end{theorem}

\begin{proof}
Suppose $x\in\R$ is not Schnorr random. Then there exists a Schnorr test $\{U_k\}_{k}$ such that $x \in \bigcap_{k} U_k$. Without loss of generality, assume that $\lambda (U_k) = 2^{-(k+2)}$. Define $\{V_n\}_{n }$ by $V_n = \bigcup_{k \geq n} U_k$. Clearly, $\{V_n\}_{n }$ is uniformly $\Sigma^0_1$. Furthermore,
\begin{gather*}
\lambda (V_n) = \lambda \left(\bigcup_{k \geq n} U_k\right) \leq \sum_{k \geq n} \lambda (U_k) = \sum_{k \geq n} 2^{-(k+2)} = 2^{-(n+1)}.
\end{gather*}
Since $\lambda (U_n)$ is computable uniformly in $n$, $\lambda (V_n)$ is also computable uniformly in $n$. Therefore, $\{V_n\}_{n }$ is a Schnorr test.

Define $I_k = [-k,k]$ for each $k$ and $f_m = \sum_{k=0}^{m} 2^{-k} \chi_{I_{k+1} \setminus V_m}$ for each $m$. Observe that
\begin{align*}
\int_{\R} f_m \ d\lambda &= \int_{\R} \sum_{k=0}^{m} 2^{-k} \chi_{I_{k+1} \setminus V_m} \ d\lambda \\
&= \sum_{k=0}^{m} 2^{-k} \lambda(I_{k+1} \setminus V_m) \\
&\leq \sum_{k=0}^{m} 2^{-k} 2(k+1) \\
&= \frac{2^{m+2} - m - 3}{2^{m-1}}.
\end{align*}
Furthermore, the sequence $\{f_m\}_{m}$ satisfies $f_m \leq f_{m+1}$ and $f_m \geq 0$ for all $m$.

We claim that the sequence $\{f_m\}_{m}$ is $L^1(\R)$-computable. To see this, note that
\begin{align*}
\norm{f_{m+1} - f_m}_1 &= \int_{\R} (f_{m+1} - f_m) \ d\lambda \\
&= \int_{\R} \left(\sum_{k=0}^{m+1} 2^{-k} \chi_{I_{k+1} \setminus V_{m+1}} - \sum_{k=0}^{m} 2^{-k} \chi_{I_{k+1} \setminus V_m}\right) \ d\lambda \\
&= \int_{\R} \left(2^{-(m+1)} \chi_{I_{m+2} \setminus V_{m+1}} - \sum_{k=0}^{m} 2^{-k} ( \chi_{I_{k+1} \setminus V_{m+1}} - \chi_{I_{k+1} \setminus V_m} )\right) \ d\lambda \\
&= 2^{-(m+1)} \lambda (I_{m+2} \setminus V_{m+1}) + \int_{\R} \sum_{k=0}^{m} 2^{-k} \chi_{I_{k+1} \cap U_m} \ d\lambda \\
& = 2^{-(m+1)} \lambda (I_{m+2} \setminus V_{m+1}) +  \sum_{k=0}^{m} 2^{-k} \lambda(I_{k+1} \cap U_m) \\
& \leq 2^{-(m+1)} 2(m+2) +  \sum_{k=0}^{m} 2^{-k} \lambda(U_m) \\
& <\frac{2(m+2)}{2^{(m+1)}} + \frac{2^{m+1} - 1}{2^{m}}\frac{1}{2^{m+2}} \\
& <\frac{2m+5}{2^{m+1}} .
\end{align*}
If $t \in \bigcap_{m} V_m$, then $t \in V_m$ for all $m$. Thus, $f_m (t) = 0$ for all $m$ so that $\lim_{m \to \infty} f_m (t) = 0$. If $t \not\in \bigcap_{m} V_m$ and $|t| \leq k+1$, then $f_m (t) = 0$ if $m \leq k+1$, and $f_m (t) = \sum_{i=0}^{k} 2^{-i} = 2-2^{-k}$ if $m > k+1$. Therefore, $\lim_{m \to \infty} f_m (t) = 2-2^{-k}$. 

Now, let $f$ be given by 
\[ 
 f(t) = \displaystyle\lim_{m \to \infty} f_m (t) =
\begin{cases}
  0  & t \in \bigcap_{m} V_m \\
  2-2^{-k} & t \not\in \bigcap_{m} V_m \text{ and } |t| \leq k
\end{cases}.
\]
Then $f=\displaystyle\sum_{k=0}^{\infty} 2^{-k} \chi_{I_{k+1} \setminus \bigcap_{n} V_n} $.

Since $\lambda (\bigcap_{n} V_n) = 0$, it follows that $f_m$ converges to $f$ pointwise almost everywhere. By the Monotone Convergence Theorem,
\begin{align*}
\int_{\R} f \ d\lambda &= \lim_{m \to \infty} \int_{\R} f_m \ d\lambda \leq \lim_{m \to \infty} \frac{2^{m+2} - m - 3}{2^{m-1}} = 8.
\end{align*}
Moreover, $f$ is $L^1(\R)$-computable because $f_m$ is $L^1(\R)$-computable uniformly in $m$ and $\norm{f - f_{2m}}_1 \leq 2^{-m}$.

We have left to show that $\displaystyle\lim\limits_{y\rightarrow0^+}P[f](x,y) \neq f(x)$.
For starters,
\begin{align*}
    \int_{\R}P_y(x-t)f(t) \ dt &\geq \int_{x-y/2}^{x+y/2}P_y(x-t)f(t) \ dt\\
    &= \int_{-y/2}^{y/2}P_y(-s)f(s+x) \ ds\\
    &= \int_{-y/2}^{y/2}P_y(s)f(s+x) \ ds
\end{align*}
Note that $|s|\leq y/2$ implies
\[
P_y(s)=\dfrac{y}{\pi(s^2+y^2)}\geq\dfrac{y}{\pi((y/2)^2+y^2)}=\dfrac{y}{\pi(5/4)y^2}=\dfrac{4}{5\pi y}.
\]
Thus,
\begin{align*}
    \int_{\R}P_y(x-t)f(t) \ dt &\geq \dfrac{4}{5\pi y}\int_{-y/2}^{y/2}f(s+x) \ ds\\
    &=\dfrac{4}{5\pi y}\int_{x-y/2}^{x+y/2}f(t) \ dt
\end{align*}
Let $K=\lfloor |x|\rfloor+1$. Then, $f(t)=2-2^{-K}$ for all $t\in[x-y/2,x+y/2]\setminus\left(\cap_nV_n\right)$. 
Fix $y>0$. Choose $m$ large enough so that $\lambda(V_m)\leq 2^{-(m+1)}\leq y/4$. Thus,
\[
\lambda\left(\left[x-\frac{y}{2},x+\dfrac{y}{2}\right]\setminus \left(\bigcap_nV_n\right)\right)\geq\lambda\left(\left[x-\frac{y}{2},x+\dfrac{y}{2}\right]\setminus V_m\right)\geq y-\frac{y}{4}=\frac{3y}{4}.
\]
It follows that
\begin{align*}
    \int_{\R}P_y(x-t)f(t) \ dt &\geq \dfrac{4}{5\pi y}\int_{x-y/2}^{x+y/2}f(t) \ dt\\
    &\geq\dfrac{4}{5\pi y}\cdot\dfrac{3y(2-2^{-K})}{4}\\
    &=\dfrac{3(2-2^{-K})}{5\pi}.
\end{align*}

Therefore, 
\[
\lim\limits_{y\rightarrow0^+}P[f](x,y) \geq\dfrac{3(2-2^{-K})}{5\pi}>0= f(x).
\]
\end{proof}

We now turn our attention to the \ml case. For context, two functions $f,g\in L^1(\R)$ are \emph{ML-equivalent} if $f(x)=g(x)$ for every \ml random $x\in\R$ \cite{miyabe13}.

\begin{theorem}[\cite{miyabe13}]\label{ml.equiv}
    Every weakly $L^1(\R)$ computable function is ML-equivalent to the difference of two \ml integral tests. Conversely, every difference of two \ml integral tests is ML-equivalent to a weakly $L^1(\R)$ computable function.
\end{theorem}

\begin{corollary}[\cite{miyabe13}]\label{ml.a.e.}
    Let $f$ be a weakly $L^1(\R)$ computable function, and let $\{f_n\}_{n\in\N}$ be a rational approximation of $f$. Then $\lim_nf_n(x)=f(x)$ for every Martin-L\"{o}f random $x\in\R$.
\end{corollary}

\begin{theorem}
    If $x\in\R$ is Martin-L\"{o}f random and $f\in L^1(\R)$ is weakly $L^1(\R)$-computable, then $\lim_{y\rightarrow0^+}P[f](x,y)=f(x)$.
\end{theorem}

\begin{proof}
    Let $x\in\R$ be \ml random, and let $f\in L^1(\R)$ be weakly $L^1(\R)$ computable. Then, there exists a computable sequence $\{f_n\}_{n}$ of compactly supported piecewise linear functions with rational vertices such that $f(t)=\lim_nf_n(t)$ for almost all $t\in\R$ and $\sum_n\norm{f_n-f_{n+1}}_1<\infty$. By Corollary \ref{ml.a.e.}, $f(t)=\lim_nf_n(t)$ for all \ml random $t\in\R$. For each $y>0$, note that
    \[
    \left|P[f](x,y)-P[f_n](x,y)\right|=\left|\int_{\R}P_y(x-t)(f-f_n)(t) \ dt\right|\leq\dfrac{1}{\pi y}\norm{f-f_n}_1.
    \]
    Since $\sum_n\norm{f_n-f_{n+1}}_1<\infty$ and $f=\lim_nf_n$ pointwise a.e., $\lim_{n\rightarrow\infty}\norm{f-f_n}_1 \linebreak = 0$. Thus, $\lim_n|P[f-f_n](x,y)|=0$ for all $(x,y)\in\UHP$.

    Now, fix $\epsilon>0$. Since $\lim_n|P[f-f_n](x,y)|=0$ for all $y>0$, there exists an $N_1\in\N$ such that for all $n\geq N_1$ and all $y>0$, $|P[f-f_n](x,y)|<\epsilon/3$. Since $f(x)=\lim_nf_n(x)$, there exists an $N_2\in\N$ such that for all $n\geq N_2$, $|f_n(x)-f(x)|<\epsilon_3$. Let $N=\max\{N_1,N_2\}$. Since $f_N$ is compactly supported, $\lim_{y\rightarrow0^+}P[f_N](x,y)=f_N(x)$. Thus, there exists a $\delta>0$ such that if $0<y<\delta$, we have $|P[f_N](x,y)-f_N(x)|<\epsilon/3$. Therefore, $0<y<\delta$ implies
    \begin{align*}
        |P[f](x,y)-f(x)|&\leq|P[f-f_N](x,y)|\\
        &\hspace{0.5in}+|P[f_N](x,y)-f_N(x)|\\
        &\hspace{1in}+|f_N(x)-f(x)|\\
        &<3\cdot\dfrac{\epsilon}{3}=\epsilon.
    \end{align*}
    Since $\epsilon>0$ was arbitrary, it follows that $\lim_{y\rightarrow0^+}P[f](x,y)=f(x)$.
\end{proof}

\begin{theorem}
    If $x\in\R$ is not Martin-L\"{o}f random, then there is a weakly $L^1(\R)$-computable $f\in L^1(\R)$ such that $\lim_{y\rightarrow0^+}P[f](x,y)\neq f(x)$.
\end{theorem}

\begin{proof}
    Suppose $x\in\R$ is not Martin-L\"{o}f random. Then there is a nested universal \ml test $\{U_n\}_{n}$ such that $x\in\bigcap_nU_n$. Let $\I_n$ be the list of all rational closed sets contained in $U_n$ for each $n$. Let $\I_n[s]$ consist of the first $s$ rational closed intervals enumerated into $U_n$.
    
    For each $I=[a,b]\in\I_n$, let $\tau_{n,I}$ be the function given by
    \[
    \tau_{n,I}(t)=\begin{cases}
        1 & \mathrm{if} \ t\in[\frac{3a+b}{4},\frac{a+3b}{4}],\\
        \frac{4}{b-a}(t-a) & \mathrm{if} \ t\in(a,\frac{3a+b}{4}),\\
        \frac{4}{a-b}(t-b) & \mathrm{if} \ t\in(\frac{a+3b}{4},b),\\
        0 & \mathrm{otherwise}.
    \end{cases}
    \]
    If $s=2n$, let $f_s=0$. If $s=2n+1$, let $f_s=\sum_{I\in\I_n[s]}\tau_{n,I}$. Let $f=\lim_sf_s$.

    Observe that, if $s=2n+1$, then
    \[
    \norm{f_s}_1=\norm{\sum_{I\in\I_n[s]}\tau_{n,I}}_1\leq\sum_{I\in\I_n[s]}\lambda(I)\leq(2n+1)\lambda(U_n)\leq\dfrac{2n+1}{2^n}
    \]
    Thus, $\lim_s\norm{f_s}_1=0$. It follows that $f=0$ a.e. Since $\{f_s\}_{s\in\N}$ is a computable sequence in $L^1(\R)$ with
    \[
    \sum_{s=0}^{\infty}\norm{f_s-f_{s+1}}\leq\sum_{n=0}^{\infty}\frac{2n+1}{2^n}<\infty,
    \]
    It follows that $f$ is weakly $L^1(\R)$-computable. 
    
    Now, note that there are infinitely many odd stages $s$ for which $f_s(x)>0$, while $f_s(x)=0$ at every even stage $s$. As a result, $\lim_sf_s(x)$ does not exist. Thus, $f(x)$ is undefined. However, since $f=0$ a.e., it follows that $P[f](t,y)=(P_y*f)(t)=0$ for all $(t,y)\in\UHP$. Therefore,
    \[
        \lim_{y\rightarrow0^+}P[f](x,y)=0\neq f(x).
    \]
\end{proof}

\section*{Acknowledgements}
The first author was supported in part by Simons Foundation Collaboration Grant \#420806, which also provided travel support to the third author. The authors would like to thank Laurent Bienvenu for many helpful discussions.

\bibliographystyle{plain}
\bibliography{analysis,bibliography,random,temp-references,universal}

\end{document}